\documentclass[reqno,a4paper]{amsart}
\usepackage{amssymb}
\usepackage{amsmath}
\begin{document} 
\newtheorem{prop}{Proposition}[section]
\newtheorem{Def}{Definition}[section] \newtheorem{theorem}{Theorem}[section]
\newtheorem{lemma}{Lemma}[section] \newtheorem{Cor}{Corollary}[section]

\title[LWP for Chern-Simons-Dirac]{\bf Low regularity solutions for Chern-Simons-Dirac systems in the temporal and Coulomb gauge}
\author[Hartmut Pecher]{
{\bf Hartmut Pecher}\\
Fakult\"at f\"ur Mathematik und Naturwissenschaften\\
Bergische Universit\"at Wuppertal\\
Gau{\ss}str.  20\\
42119 Wuppertal\\
Germany\\
e-mail {\tt pecher@math.uni-wuppertal.de}}
\date{}

\begin{abstract}
We prove low regularity local well-posedness results in Bourgain-Klainerman-Machedon spaces for the Chern-Simons-Dirac system in the temporal gauge and the Coulomb gauge. Under slightly stronger assumptions on the data we also obtain "unconditional" uniqueness in the natural solution spaces.
\end{abstract}
\maketitle
\renewcommand{\thefootnote}{\fnsymbol{footnote}}
\footnotetext{\hspace{-1.5em}{\it 2000 Mathematics Subject Classification:} 
35Q40, 35L70 }
\normalsize 
\setcounter{section}{0}
\section{Introduction and main results}
\noindent Consider the Chern-Simons-Dirac system in two space dimensions :
\begin{align}
\label{1}
i \partial_t \psi + i \alpha^j \partial_j \psi & = m \beta \psi - \alpha^{\mu} A_{\mu} \psi\\
\label{2}
\partial_{\mu} A_{\nu} - \partial_{\nu} A_{\mu} & = -2 \epsilon_{\mu \nu \lambda} \langle \psi , \alpha^{\lambda} \psi \rangle
\end{align}
with initial data
\begin{equation}
\label{2'}
 \psi(0) = \psi_0 \quad , \quad  A_{\mu}(0) = a_{\mu} \, , 
\end{equation}
where we use the convention that repeated upper and lower indices are summed, Latin indices run over 1,2 and Greek indices over 0,1,2 with Minkowski metric of signature ($+$,$-$,$-$).
Here $\psi : {\mathbb R}^{1+2} \to {\mathbb C}^2$ , $A_{\nu} : {\mathbb R}^{1+2} \to {\mathbb R}$ , $m \in {\mathbb R}$ . 
$\alpha^1,\alpha^2, \beta$ are hermitian ($ 2 \times 2$)-matrices satisfying 
$\beta^2 = 
(\alpha^1)^2 = (\alpha^2)^2 = I $ , $ \alpha^j \beta + \beta \alpha^j = 0, $  $ 
\alpha^j \alpha^k + \alpha^k \alpha^j = 2 \delta^{jk} I $ , $ \alpha^0 = I $ . 
$\langle \cdot,\cdot \rangle $ denotes the ${\mathbb C}^2$ - scalar product. A 
particular representation is given by \\
$ \alpha^1 = \left( \begin{array}{cc}
0 & 1  \\ 
1 & 0  \end{array} \right)$ 
 , $ \alpha^2 = \left( \begin{array}{cc}
0 & -i  \\
i & 0  \end{array} \right)$ , $ \beta = \left( \begin{array}{cc}
1 & 0  \\
0 & -1  \end{array} \right)$ . \\
$\epsilon_{\mu \nu \lambda}$ is the totally skew-symmetric tensor with $\epsilon_{012} = 1$ . 

This model was proposed by Cho, Kim and Park \cite{CKP} and Li and Bhaduri \cite{LB}. 

The equations are invariant under the gauge transformations
$$ A_{\mu} \rightarrow A'_{\mu} = A_{\mu} + \partial_{\mu} \chi \, , \, \psi \rightarrow \psi' = e^{i\chi} \phi  \, . $$
The most common gauges are the Coulomb gauge $\partial^j A_j =0$ , the Lorenz gauge $\partial^{\mu} A_{\mu} = 0$ and the temporal gauge $A_0 = 0$. 

Local well-posedness for data with minimal regularity assumptions was shown by Huh \cite{H} in the Lorenz gauge for data $\psi_0 \in H^{\frac{5}{8}} $ , $a_{\mu} \in H^{\frac{1}{2}}$ using a null structure, in the Coulomb gauge for $\psi_0 \in H^{\frac{1}{2}+\epsilon}$ , $ a_i \in L^2$, and in temporal gauge for $\psi_0 \in H^{\frac{3}{4}+\epsilon},$  $a_j \in H^{\frac{3}{4}+\epsilon} + L^2$ , both without using a null structure. The result in Lorenz gauge was improved by Huh-Oh \cite{HO} where the regularity of the data was lowered down to  $\psi_0 \in H^s$ , $a_{\mu} \in H^s$ with $s > \frac{1}{4}$ . Their proof relies also on  a null structure in the nonlinear terms of the Dirac equation as well as the wave equation. They apply a Picard iteration in Bourgain-Klainerman-Machedon spaces $X^{s,b}$ , which implies uniqueness in these spaces. Independently Okamoto \cite{O} proved a similar result in Lorenz as well as Coulomb gauge also using a null structure of the system. The methods of Okamoto and Huh-Oh are different. Okamoto reduces the problem to a single Dirac equation with cubic nonlinearity for $\psi$, which does not contain $A_{\mu}$ any longer. From a solution $\psi$ of this equation the potentials $A_{\mu}$ can be constructed by solving a wave equation in Lorenz gauge and an elliptic equation in Coulomb gauge. Huh-Oh on the other hand directly solve a coupled system of a Dirac equation for $\psi$ and a wave equation for $A_{\mu}$. Recently Bournaveas-Candy-Machihara \cite{BCM} proved local well-posedness in Coulomb gauge under similar regularity assumptions without use of a null structure. Their proof relies on a bilinear Strichartz estimate given by Klainerman-Tataru \cite{KT}.

A low regularity local well-posedness result in  temporal gauge was given by Tao \cite{T1} for the Yang-Mills equations.

In the present paper we consider the temporal gauge as well as the Coulomb gauge. In temporal gauge we improve the result of Huh \cite{H} to data $\psi_0 \in H^s$,  $a_j \in H^{s+\frac{1}{8}}$ with $s > \frac{3}{8}$. We use Bourgain-Klainerman-Machedon spaces $X^{s,b}$ adapted to the phase functions $\tau \pm |\xi|$ on one hand and $\tau$ on the other hand. We decompose $A_j$ into its divergence-free part $A^{df}_j$ and its curl-free part $A^{cf}_j$. The main problem here is that there seems to be no null structure in the nonlinearity $A^{cf}_j \alpha^j \psi$ in the Dirac equation whereas in Lorenz gauge $A^{cf}_{\mu} \alpha^{\mu} \psi$ has such a null structure. In fact all the other terms possess such a null structure. However we are not able to use it for an improvement of our result. We apply the bilinear estimates in wave-Sobolev spaces established in d'Ancona-Foschi-Selberg \cite{AFS} which rely on Strichartz estimates. Morover we use a variant of an estimate for the $L^6_x L^2_t$ - norm for the solution of the wave equation which goes back to Tataru and Tao. When applying this estimate we partly follow Tao's arguments in the case of the Yang-Mills equations \cite{T1}. We prove existence and uniqueness in $X^{s,b}$ - spaces first (Theorem \ref{Theorem 1.1}). Then we prove unconditional uniqueness under the stronger assumption $s > \frac{19}{40}$ (Theorem \ref{Theorem 1.2}) by using an idea of Zhou \cite{Z}.

In Coulomb gauge we make the same regularity asasumptions as Okamoto \cite{O} and Bournaveas-Candy-Machihara \cite{BCM}, namely $\psi_0 \in H^{\frac{1}{4}+\epsilon}$, and also reduce the problem to a single Dirac equation with cubic nonlinearity. We give a short (alternative) proof of local well-posedness in $X^{s,b}$ - spaces without use of a null structure (Theorem {\ref{Theorem 1.1'}) using d'Ancona-Foschi-Selberg \cite{AFS} (cf. Proposition \ref{Prop. 1.1}). We also prove unconditional uniqueness in the space $\psi \in C^0([0,T],H^s)$ under the assumption $s > \frac{1}{3}$ (Theorem \ref{Theorem 1.2'}).

We first give some notation. We denote the Fourier transform with respect to space and time by $\,\widehat{}$ . The operator
$|\nabla|^{\alpha}$ is defined by ${\mathcal F}(|\nabla|^{\alpha} f)(\xi) = |\xi|^{\alpha} ({\mathcal F}f)(\xi)$, where ${\mathcal F}$ is the Fourier transform, and similarly $ \langle \nabla \rangle^{\alpha}$ , where $\langle \cdot \rangle := (1+|\cdot|^2)^{\frac{1}{2}}$. The inhomogeneous and homogeneous Sobolev spaces are denoted by $H^{s,p}$ and $\dot{H}^{s,p}$, respectively. For $p=2$ we simply denote them by $H^s$ and $\dot{H}^s$. We repeatedly use the Sobolev embeddings $\dot{H}^{s,p} \hookrightarrow L^q$ for  $1<p\le q < \infty$ and $\frac{1}{q} = \frac{1}{p}-\frac{s}{2}$, and also $\dot{H}^{1+} \cap \dot{H}^{1-} \hookrightarrow  L^{\infty}$ in two space dimensions. 
$a+ := a + \epsilon$ for a sufficiently small $\epsilon >0$ , so that $a<a+<a++$ , and similarly $a--<a-<a$ .

We define the standard spaces $X^{s,b}_{\pm}$ of Bourgain-Klainerman-Machedon type belonging to the half waves as the completion of the Schwarz space  $\mathcal{S}({\mathbb R}^3)$ with respect to the norm
$$ \|u\|_{X^{s,b}_{\pm}} = \| \langle \xi \rangle^s \langle  \tau \pm |\xi| \rangle^b \widehat{u}(\tau,\xi) \|_{L^2_{\tau \xi}} \, . $$ 
Similarly we define the wave-Sobolev spaces $X^{s,b}_{|\tau|=|\xi|}$ with norm
$$ \|u\|_{X^{s,b}_{|\tau|=|\xi|}} = \| \langle \xi \rangle^s \langle  |\tau| - |\xi| \rangle^b \widehat{u}(\tau,\xi) \|_{L^2_{\tau \xi}}  $$ and also $X^{s,b}_{\tau =0}$ with norm 
$$\|u\|_{X^{s,b}_{\tau=0}} = \| \langle \xi \rangle^s \langle  \tau  \rangle^b \widehat{u}(\tau,\xi) \|_{L^2_{\tau \xi}} \, .$$
We also define $X^{s,b}_{\pm}[0,T]$ as the space of the restrictions of functions in $X^{s,b}_{\pm}$ to $[0,T] \times \mathbb{R}^2$ and similarly $X^{s,b}_{|\tau| = |\xi|}[0,T]$ and $X^{s,b}_{\tau =0}[0,T]$. We frequently use the obvious embeddings $X^{s,b}_{|\tau|=|\xi|} \hookrightarrow X^{s,b}_{\pm}$      for $b \le 0$ and $X^{s,b}_{\pm} \hookrightarrow X^{s,b}_{|\tau|=|\xi|} $ for $b \ge 0$. 
\vspace{0.5em}

We now formulate our main results in the case of the temporal gauge.
\begin{theorem}
\label{Theorem 1.1}
Let $\epsilon > 0$ and $s > \frac{3}{8}$. The Chern-Simons-Dirac system (\ref{1}),(\ref{2}),(\ref{2'}) in temporal gauge $A_0=0$ with data $\psi_0 \in H^s({\mathbb R}^2)$ , $a_j \in H^{s+\frac{1}{8}}({\mathbb R}^2)$, satisfying the compatability condition $\partial_1 a_2- \partial_2 a_1 = -2 \langle \psi_0, \psi_0 \rangle$, has a local solution 
$$ \psi \in C^0([0,T],H^s({\mathbb R}^2)) \quad , \quad |\nabla|^{\epsilon} A_j \in C^0([0,T],H^{s+\frac{1}{8}-\epsilon}({\mathbb R}^2)) \, . $$
More precisely $\psi = \psi_+ + \psi_-$ with $\psi_{\pm} \in X^{s,\frac{1}{2}+}_{\pm}[0,T]$. If 
$ A = A^{df} + A^{cf}  $
is the decomposition into its divergence-free part and its "curl-free"
 part, where
\begin{align*}
A^{df} & = (-\Delta)^{-1}(\partial_2(\partial_1A_2-\partial_2 A_1),\partial_1(\partial_2 A_1-\partial_1 A_2)) \, , \\
A^{cf} & = -(-\Delta)^{-1}(\partial_1(\partial_1 A_1 + \partial_2 A_2),\partial_2(\partial_1 A_1 + \partial_2 A_2)) = -(-\Delta)^{-1} \nabla div A \, ,
\end{align*}
one has
$$ A^{cf} \in X^{s+\frac{1}{8},\frac{1}{2}+}_{\tau =0} [0,T] \quad , \quad |\nabla|^{\epsilon} A^{df} \in X^{s+\frac{3}{8}-\epsilon,\frac{1}{2}+}_{|\tau|=|\xi|}[0,T] $$
and in these spaces uniqueness holds. Moreover we have $\psi_{\pm} \in X^{s,1}_{\pm} [0,T]$.
\end{theorem} 
{\bf Remark:} The Chern-Simons-Dirac system is invariant under the scaling
$$ \psi^{(\lambda)}(t,x) = \lambda\psi(\lambda t,\lambda x) \, , \, A^{(\lambda)}(t,x) = \lambda A_{\mu}(\lambda t, \lambda x) \, . $$
Thus in 2+1 dimensions the scaling critical Sobolev exponent is $s=0$ , i.e. $\psi_0$ , $a_{\mu} \in H^s = L^2$ . In Lorenz gauge Huh-Oh \cite{HO} remarked that their result $s > \frac{1}{4}$ is probably optimal in view of Zhou \cite{Z1}, who proved that is the case for a system of nonlinear wave equations with nonlinearities, which fulfill a null condition. In our case of the temporal gauge however the system is reduced to a coupled system of a wave equation for $\psi$ and a transport equation for $A^{cf}$ where null conditions seem to be not useful  because they are only adapted for wave equations. Nevertheless it would be desirable to improve our result  to $s > \frac{1}{4}$ for $\psi_0$ and $a_j$ . 

\begin{theorem}
\label{Theorem 1.2} Let the assumptions of Theorem \ref{Theorem 1.1} be fulfilled.
If moreover $s > \frac{19}{40}$,  the solution of (\ref{1}),(\ref{2}),(\ref{2'}) is unique in the space
$ \psi \in C^0([0,T],H^s({\mathbb R}^2)) \, , \, A^{cf} \in C^0([0,T],H^{s+\frac{1}{8}}({\mathbb R}^2)) \, , \,  |\nabla|^{\epsilon} A^{df} \in C^0([0,T],H^{s+\frac{3}{8}-\epsilon}({\mathbb R}^2)) \, . $
\end{theorem}
\vspace{0.5em}
Consider now the Coulomb gauge condition $\partial_j A^j = 0$. In this case one easily checks using (\ref{2}) that the potentials $A_{\mu}$ satisfy the elliptic equations
\begin{equation}
\label{1.10}
A_0 = \Delta^{-1}(\partial_2 \langle \psi,\alpha^1 \psi \rangle - \partial_1 \langle \psi,\alpha^2 \psi \rangle) \, , \, A_1 = \Delta^{-1} \partial_2 \langle \psi,\psi \rangle \, , \, A_2 = - \Delta^{-1} \partial_1 \langle \psi,\psi \rangle \, . 
\end{equation}
Inserting this into (\ref{1}) we obtain
\begin{equation}
\label{1.11}
i \partial_t \psi + i \alpha^j \partial_j \psi  = m \beta \psi + N(\psi,\psi,\psi) \, ,
\end{equation}
where
\begin{align*}
 &N(\psi_1,\psi_2,\psi_3) \\
 &= \Delta^{-1}\left( \partial_1 \langle  \psi_1,\alpha_2 \psi_2 \rangle - \partial_2 \langle \psi_1,\alpha_1 \psi_2 \rangle + \partial_2 \langle \psi_1,\psi_2 \rangle \alpha_1 - \partial_1 \langle \psi_1,\psi_2 \rangle \alpha_2 \right) \psi_3 \, . 
\end{align*}
In the sequel we consider this nonlinear Dirac equation with initial condition
\begin{equation}
\label{1.12}
\psi(0) = \psi_0 \, .
\end{equation}

Using an idea of d'Ancona - Foschi -Selberg \cite{AFS1} we simplify  (\ref{1.11}) by 
considering the projections onto the one-dimensional eigenspaces of the 
operator 
$-i \alpha \cdot \nabla = -i \alpha^j \partial_j$ belonging to the eigenvalues $ \pm |\xi|$. These 
projections are given by $\Pi_{\pm} =$ $\Pi_{\pm}(D)$, where  $ D = 
\frac{\nabla}{i} $ and $\Pi_{\pm}(\xi) = \frac{1}{2}(I 
\pm \frac{\xi}{|\xi|} \cdot \alpha) $. Then $ 
-i\alpha \cdot \nabla = |D| \Pi_+(D) - |D| \Pi_-(D) $ and $ \Pi_{\pm}(\xi) \beta
= \beta \Pi_{\mp}(\xi) $. Defining $ \psi_{\pm} := \Pi_{\pm}(D) \psi$  , the Dirac  equation can be rewritten as
\begin{equation}
\label{4'}
(-i \partial_t \pm |D|)\psi_{\pm}  =  m\beta \psi_{\mp} + \Pi_{\pm}N(\psi_+ + \psi_-,\psi_+ + \psi_-, \psi_+ + \psi_-) \, . 
\end{equation}
The initial condition is transformed into
\begin{equation}
\label{6'}
\psi_{\pm}(0) = \Pi_{\pm}\psi_0 \, .
\end{equation}

We now formulate our results in the case of the Coulomb gauge.

\begin{theorem}
\label{Theorem 1.1'}
Assume $\psi_0 \in H^s({\mathbb R}^2)$ with $s > \frac{1}{4}$. Then (\ref{1.11}),(\ref{1.12}) is locally well-posed in $H^s({\mathbb R}^2)$. More precisely there exists $T>0$ , such that there exists a unique solution $\psi = \psi_+ + \psi_-$ with $\psi_{\pm} \in X^{s,\frac{1}{2}+}_{\pm}[0,T]$. This solution belongs to $C^0([0,T],H^s({\mathbb R}^2))$.
\end{theorem}

The unconditional uniqueness result is the following
\begin{theorem}
\label{Theorem 1.2'}
Assume $ \psi_0 \in H^s({\mathbb R}^2)$ with $ s > \frac{1}{3}$. The solution of (\ref{1.11}),(\ref{1.12}) is unique in $C^0([0,T],H^s({\mathbb R}^2))$.
\end{theorem}
Fundamental for the proof of our theorems are the following bilinear estimates in wave-Sobolev spaces which were proven by d'Ancona, Foschi and Selberg in the two dimensional case $n=2$ in \cite{AFS} in a more general form which include many limit cases which we do not need.
\begin{prop}
\label{Prop. 1.1}
Let $n=2$. The estimate
$$\|uv\|_{X_{|\tau|=|\xi|}^{-s_0,-b_0}} \lesssim \|u\|_{X^{s_1,b_1}_{|\tau|=|\xi|}} \|v\|_{X^{s_2,b_2}_{|\tau|=|\xi|}} $$ 
holds, provided the following conditions hold:
\begin{align*}
\nonumber
& b_0 + b_1 + b_2 > \frac{1}{2} \\
\nonumber
& b_0 + b_1 \ge 0 \\
\nonumber
& b_0 + b_2 \ge 0 \\
\nonumber
& b_1 + b_2 \ge 0 \\
\nonumber
&s_0+s_1+s_2 > \frac{3}{2} -(b_0+b_1+b_2) \\
\nonumber
&s_0+s_1+s_2 > 1 -\min(b_0+b_1,b_0+b_2,b_1+b_2) \\
\nonumber
&s_0+s_1+s_2 > \frac{1}{2} - \min(b_0,b_1,b_2) \\
\nonumber
&s_0+s_1+s_2 > \frac{3}{4} \\
 &(s_0 + b_0) +2s_1 + 2s_2 > 1 \\
\nonumber
&2s_0+(s_1+b_1)+2s_2 > 1 \\
\nonumber
&2s_0+2s_1+(s_2+b_2) > 1 \\
\nonumber
&s_1 + s_2 \ge \max(0,-b_0) \\
\nonumber
&s_0 + s_2 \ge \max(0,-b_1) \\
\nonumber
&s_0 + s_1 \ge \max(0,-b_2)   \, .
\end{align*}
\end{prop}
\vspace{0.5em}
Another decisive tool are the estimates for the wave equation in the following proposition.
\begin{prop}
The following estimates hold
\begin{align}
\label{15}
\|u\|_{L^6_{xt}} & \lesssim \|u\|_{X^{\frac{1}{2},\frac{1}{2}+}_{|\tau|=|\xi|}} \, , \\
\label{16}
\|u\|_{L^p_x L^2_t} & \lesssim \|u\|_{X^{\frac{1}{2}-\frac{2}{p},\frac{1}{2}+}_{|\tau|=|\xi|}} \quad \mbox{for} \,\, 6 \le p < \infty \, ,\\
\label{16'}
\mbox{especially} \quad \|u\|_{L^6_x L^2_t} &\lesssim \|u\|_{X^{\frac{1}{6},\frac{1}{2}+}_{|\tau|=|\xi|}} \, , \\
\label{16''}
\|u\|_{L^{\infty}_x L^2_t} &\lesssim \|u\|_{X^{\frac{1}{2}+,\frac{1}{2}+}_{|\tau|=|\xi|}} \, , \\
\label{16'''}
\|u\|_{L^{\infty}_x L^{2+}_t} &\lesssim \|u\|_{X^{\frac{1}{2}+,\frac{1}{2}+}_{|\tau|=|\xi|}} \, , \\
\label{18}
\|u\|_{L^6_x L^{2+}_t} &\lesssim \|u\|_{X^{\frac{1}{6}+,\frac{1}{2}+}_{|\tau|=|\xi|}} \, , \\
\label{17}
\|u\|_{L^4_x L^{2+}_t} &\lesssim \|u\|_{X^{\frac{1}{8}+,\frac{3}{8}+}_{|\tau|=|\xi|}} \, , \\
\label{19} 
\|u\|_{L^p_x L^{2+}_t} & \lesssim \|u\|_{X^{\frac{1}{2}-\frac{2}{p}+,\frac{1}{2}+}_{|\tau|=|\xi|}}  \quad \mbox{for} \,\, 6 \le p < \infty \, . 
\end{align}
\end{prop}
\begin{proof}
(\ref{15}) is the standard Strichartz estimate combined with the transfer principle. Concerning (\ref{16}) we use \cite{KMBT} (appendix by D. Tataru) Thm. B2:
$$ \|{\mathcal F}_t u \|_{L^2_{\tau} L^6_x} \lesssim \|u_0\|_{\dot{H}^{\frac{1}{6}}_x} \, , $$
if $u=e^{it |\nabla|}u_0$ , and ${\mathcal F}$ denotes the Fourier transform with respect to time. This implies by Plancherel, Minkowski's inequality and Sobolev's embedding theorem
$$ \|u\|_{L^p_x L^2_t} = \|{\mathcal F}_t u\|_{L^p_x L^2_{\tau}} \lesssim \|{\mathcal F}_t u\|_{L^2_{\tau} L^p_x} \lesssim \|{\mathcal F}_t u\|_{L^2_{\tau} H^{\frac{1}{3}-\frac{2}{p},6}_x} \lesssim \|u_0\|_{H^{\frac{1}{2}-\frac{2}{p},2}_x} \, . $$
The transfer principle gives (\ref{16}). (\ref{16''}) follows similarly using $H^{\frac{1}{3}+,6}_x  \hookrightarrow L^{\infty}_x$ . (\ref{18}) is obtained by interpolation between (\ref{16'}) and (\ref{15}) , and (\ref{17}) by interpolation between (\ref{18}) and the trivial identity $\|u\|_{L^2_{xt}} = \|u\|_{X^{0,0}_{|\tau|=|\xi|}}$. Moreover we obtain (\ref{16'''}) and (\ref{19}) by interpolation between (\ref{16''}) and (\ref{16}), resp. , and the estimate $\|u\|_{L^{\infty}_{xt}} \lesssim \|u\|_{X^{1+,\frac{1}{2}+}_{|\tau|=|\xi|}}$ .
\end{proof}

\section{Reformulation of the problem in temporal gauge}
Imposing the temporal gauge condition $A_0 =0$ the system (\ref{1}),(\ref{2}) is equivalent to
\begin{align}
\label{3}
&i \partial_t \psi + i \alpha^j \partial_j \psi  = m \beta \psi - \alpha^{j} A_{j} \psi\\
\label{4}
&\partial_t A_1 = -2 \langle \psi , \alpha^2 \psi \rangle \quad  , \quad \partial_t A_2 = 2 \langle \psi , \alpha^1 \psi \rangle \\
\label{5}
&\partial_1 A_2 - \partial_2 A_1  = -2 \langle \psi, \psi \rangle \, .
\end{align}
We first show that (\ref{5}) is fulfilled for any solution of (\ref{3}),(\ref{4}), if it holds initially, i.e., if the following compatability condition holds:
\begin{equation}
\partial_1 A_2(0) - \partial_2 A_1(0) = -2 \langle \psi(0),\psi(0) \rangle \, ,
\end{equation}
which we assume from now on. Indeed one easily calculates using (\ref{3}):
\begin{equation}
\label{8}
\partial_t \langle \psi,\psi \rangle = -\partial_j \langle \psi, \alpha^j \psi \rangle \, ,
\end{equation}
which implies by (\ref{4})
$$ \partial_t(\partial_1 A_2 - \partial_2 A_1) = 2 \partial_j \langle \psi,\alpha^j \psi \rangle = -2 \partial_t \langle \psi, \psi \rangle \, , $$
so that (\ref{5}) holds, if it holds initially. Thus we only have to solve (\ref{3}) and (\ref{4}).

We decompose $A=(A_1,A_2)$ into its divergence-free part $A^{df}$ and its "curl-free" part $A^{cf}$, namely $A=A^{df} + A^{cf}$, where
\begin{align*}
A^{df} & = (-\Delta)^{-1}(\partial_2(\partial_1A_2-\partial_2 A_1),\partial_1(\partial_2 A_1-\partial_1 A_2)) \, , \\
A^{cf} & = -(-\Delta)^{-1}(\partial_1(\partial_1 A_1 + \partial_2 A_2),\partial_2(\partial_1 A_1 + \partial_2 A_2)) = -(-\Delta)^{-1} \nabla div A \, .
\end{align*}
Then (\ref{5}) and (\ref{4}) imply
\begin{align}
\label{6}
A^{df} &= -2(-\Delta)^{-1} (\partial_2 \langle \psi,\psi \rangle,-\partial_1 \langle \psi, \psi \rangle) \,, \\
\label{7}
\partial_t A_j^{cf} & = -2(-\Delta)^{-1} \partial_j(\partial_2 \langle \psi,\alpha^1 \psi \rangle - \partial_1 \langle \psi,\alpha^2 \psi \rangle) \, .
\end{align}
Reversely, defining $A=A^{df}+A^{cf}$, we show that our new system (\ref{3}),(\ref{6}),(\ref{7}) implies (\ref{3}),(\ref{4}),(\ref{5}), so that both systems are equivalent. It only remains to show that (\ref{4}) holds. By (\ref{6}),(\ref{7}),(\ref{8}) we obtain
\begin{align*}
\partial_t A_1 & = \partial_t A_1^{df} + \partial_t A_1^{cf} \\
& = -2(-\Delta)^{-1}\big( \partial_2 \partial_t \langle \psi,\psi \rangle + \partial_1(\partial_2 \langle \psi,\alpha^1 \psi \rangle - \partial_1 \langle \psi, \alpha^2 \psi \rangle ) \big) \\
& =  2(-\Delta)^{-1}\big( \partial_2 \partial_j \langle \psi,\alpha^j \psi \rangle - \partial_1(\partial_2 \langle \psi,\alpha^1 \psi \rangle - \partial_1 \langle \psi, \alpha^2 \psi \rangle ) \big) \\
& = 2(-\Delta)^{-1}(\partial_2^2 + \partial_1^2) \langle \psi,\alpha^2 \psi \rangle = -2 \langle \psi, \alpha^2 \psi \rangle
\end{align*}
and similarly
$$ \partial_t A_2 = 2 \langle \psi,\alpha^1 \psi \rangle \, . $$
In the same way in which we obtained (\ref{4'}) the Dirac equation (\ref{3}) can be rewritten as
\begin{equation}
\label{9}
(- i\partial_t \pm|\nabla|)\psi_{\pm} = -m\beta \psi_{\mp} - \Pi_{\pm}(\alpha^j A_j \psi) \, , 
\end{equation}
where $A_j = A_j^{df} + A_j^{cf}$ , and in (\ref{6}),(\ref{7}) and (\ref{9}) we replace $\psi$ by $\psi_+ + \psi_-$.

\section{Proof of Theorem \ref{Theorem 1.1}}
Taking the considerations of the previous section into account Theorem \ref{Theorem 1.1} reduces to the following proposition and its corollary.
\begin{prop}
\label{Prop. 2.1}
Let $\epsilon >0$ and $s > \frac{3}{8}$. There exists $T>0$ such that the system (\ref{6}),(\ref{7}),(\ref{9}) has a unique local solution $\psi_{\pm} \in X^{s,\frac{1}{2}+}_{\pm}[0,T] \, , \, A^{cf} \in X^{s+\frac{1}{8},\frac{1}{2}+}_{\tau =0} [0,T] $. Moreover $A^{df}$ satisfies $|\nabla|^{\epsilon} A^{df}_j \in X^{s+\frac{3}{8}-\epsilon,\frac{1}{2}+}_{|\tau|=|\xi|}[0,T]$ and $\psi_{\pm} \in X^{s,1}_{\pm}[0,T]$.
\end{prop}
\begin{Cor}
\label{Cor. 2.1}
The solution satisfies $\psi \in C^0([0,T],H^s)$ , $A^{cf} \in C^0([0,T],H^{s+\frac{1}{8}})$, $|\nabla|^{\epsilon} A^{df} \in C^0([0,T],H^{s+\frac{3}{8}-\epsilon})$ .
\end{Cor}
\begin{proof}[Proof of Proposition \ref{Prop. 2.1}]
We want to apply a Picard iteration. For the Cauchy problem for the Dirac equation 
$$(- i\partial_t \pm|\nabla|)\psi_{\pm} = F_{\pm} \quad , \quad \psi_{\pm}(0) = \psi_{\pm 0} $$
we use the well-known estimate (cf. e.g. \cite{GTV})
$$\|\psi_{\pm}\|_{X^{s,b}_{\pm}[0,T]} \lesssim \|\psi_{\pm 0}\|_{H^s} + T^{1+b'-b} \|F_{\pm}\|_{X^{s,b'}_{\pm}[0,T]} \, , $$
which holds for $0<T \le 1$ , $-\frac{1}{2}<b' \le0 \le b \le b'+1$ , $s \in{\mathbb R}$ .
Thus by standard arguments it suffices to show the following estimates for the right hand side of the Dirac equation (\ref{9}):
\begin{align}
\label{2.1}
\|A_j^{cf} \alpha^j \psi\|_{X^{s,-\frac{1}{2} ++}_{|\tau|=|\xi|}} & \lesssim \|A^{cf}\|_{X^{s+\frac{1}{8},\frac{1}{2}+}_{\tau=0}} \|\psi\|_{X^{s,\frac{1}{2}+}_{|\tau|=|\xi|}} \, , \\
\label{2.2}
 \|A_j^{df} \alpha^j \psi\|_{X^{s,-\frac{1}{2} ++}_{|\tau|=|\xi|}} & \lesssim \||\nabla|^{\epsilon} A_j^{df}\|_{X^{s+\frac{3}{8}-\epsilon,\frac{1}{2}+}_{|\tau|=|\xi|}} \|\psi\|_{X^{s,\frac{1}{2}+}_{|\tau|=|\xi|}} \, , \\
\label{2.3}
\||\nabla|^{\epsilon} A_j^{df}\|_{X^{s+\frac{3}{8}-\epsilon,\frac{1}{2}+}_{|\tau|=|\xi|}} &\lesssim 
\|\psi\|_{X^{s,\frac{1}{2}++}_{|\tau|=|\xi|}}^2 \, .
\end{align} 
Similarly, for the right hand side of (\ref{7}) we need
\begin{equation}
\label{2.4}
\| \langle \psi,\alpha^j \psi \rangle \|_{X^{s+\frac{1}{8},-\frac{1}{2}++}_{\tau=0}} \lesssim \|\psi\|^2_{X^{s,\frac{1}{2}+}_{|\tau|=|\xi|}} \, .
\end{equation}
{\bf Proof of (\ref{2.1}):}
We even prove the estimate with $X^{s,-\frac{1}{2}++}_{|\tau|=|\xi|}$ replaced by
$X^{s,0}_{|\tau|=|\xi|}$ on the left hand side. It reduces to
\begin{align*}
 \int_* \frac{\widehat{u}_1(\tau_1,\xi_1)}{\langle  \xi_1\rangle^{s+\frac{1}{8}} \langle \tau_1 \rangle^{\frac{1}{2}+}} 
\frac{\widehat{u}_2(\tau_2,\xi_2)}{\langle \xi_2 \rangle^s\langle |\tau_2| - |\xi_2|\rangle^{\frac{1}{2}+}}\langle \xi_3 \rangle^s
\widehat{u}_3(\tau_3,\xi_3) d\xi d\tau
\lesssim \prod_{i=1}^3 \|u_i\|_{L^2_{xt}} \, ,
\end{align*}
where * denotes integration over $\xi = (\xi_1,\xi_2,\xi_3) , \tau=(\tau_1,\tau_2,\tau_3)$ with $\xi_1+\xi_2+\xi_3=0$ and $\tau_1+\tau_2+\tau_3 =0$. We assume here and in the following without loss of generality that the Fourier transforms are nonnnegative. \\
Case 1: $|\xi_1| \ge |\xi_2|$ $\Rightarrow$ $\langle \xi_3 \rangle^s \lesssim \langle \xi_1 \rangle^s $ .\\
It suffices to show
\begin{align*}
 \int_* \frac{\widehat{u}_1(\tau_1,\xi_1)}{ \langle \tau_1 \rangle^{\frac{1}{2}+}} 
\frac{\widehat{u}_2(\tau_2,\xi_2)}{\langle \xi_2 \rangle^{s+\frac{1}{8}}\langle |\tau_2| - |\xi_2|\rangle^{\frac{1}{2}+}}
\widehat{u}_3(\tau_3,\xi_3) d\xi d\tau 
\lesssim \prod_{i=1}^3 \|u_i\|_{L^2_{xt}} \, .
\end{align*}
This follows under the assumption $s>\frac{3}{8}$ from the estimate
\begin{align}
\nonumber
\Big|\int v_1 v_2 v_3 dx dt \Big| & \lesssim \|v_1\|_{L^2_x L^{\infty}_t} \|v_2\|_{L^{\infty}_x L^2_t} \|v_3\|_{L^2_x L^2_t} \\
\label{20}
&\lesssim \|v_1\|_{X^{0,\frac{1}{2}+}_{\tau=0}} \|v_2\|_{X^{\frac{1}{2}+,\frac{1}{2}+}_{|\tau|=|\xi|}} 
\|v_3\|_{X^{0,0}_{|\tau|=|\xi|}} \, ,
\end{align}
where we used (\ref{16''}).\\
Case 2: $|\xi_2| \ge |\xi_1|$ $\Rightarrow$ $\langle \xi_3 \rangle^s \lesssim \langle \xi_2 \rangle^s $.\\
In this case the desired estimate follows from
\begin{equation}
\label{21}
\int_* m(\xi_1,\xi_2,\xi_3,\tau_1,\tau_2,\tau_3) \widehat{u}_1(\xi_1,\tau_1)  \widehat{u}_2(\xi_2,\tau_2) \widehat{u}_3(\xi_3,\tau_3) d\xi d\tau \lesssim \prod_{i=1}^3 \|u_i\|_{L^2_{xt}} \, , 
\end{equation}
where 
$$ m = \frac{1}{ \langle |\tau_2| - |\xi_2|\rangle^{\frac{1}{2}+}  \langle \xi_1 \rangle^{\frac{1}{2}+}\langle \tau_1 \rangle^{\frac{1}{2}+}} \, .$$
The following argument is closely related to the proof of a similar estimate in  \cite{T1}.\\
By two applications of the averaging principle (\cite{T}, Prop. 5.1) we may replace $m$ by
$$ m' = \frac{ \chi_{||\tau_2|-|\xi_2||\sim 1} \chi_{|\tau_1| \sim 1}}{ \langle \xi_1 \rangle^{\frac{1}{2}+}} \, . $$
Let now $\tau_2$ be restricted to the region $\tau_2 =T + O(1)$ for some integer $T$. Then $\tau_3$ is restricted to $\tau_3 = -T + O(1)$, because $\tau_1 + \tau_2 + \tau_3 =0$, and $\xi_2$ is restricted to $|\xi_2| = |T| + O(1)$. The $\tau_3$-regions are essentially disjoint for $T \in {\mathbb Z}$ and similarly the $\tau_2$-regions. Thus by Schur's test (\cite{T}, Lemma 3.11) we only have to show
\begin{align*}
 &\sup_{T \in {\mathbb Z}} \int_* \frac{ \chi_{\tau_3=-T+O(1)} \chi_{\tau_2=T+O(1)} \chi_{|\tau_1|\sim 1} \chi_{|\xi_2|=|T|+O(1)}}{\langle \xi_1 \rangle^{\frac{1}{2}+}}\cdot \\
 & \hspace{14em} \cdot\widehat{u}_1(\xi_1,\tau_1) \widehat{u}_2(\xi_2,\tau_2)
\widehat{u}_3(\xi_3,\tau_3) d\xi d\tau \lesssim \prod_{i=1}^3 \|u_i\|_{L^2_{xt}} \, . 
\end{align*}
The $\tau$-behaviour of the integral is now trivial, thus we reduce to
\begin{equation}
\label{50}
\sup_{T \in {\mathbb N}} \int_{\sum_{i=1}^3 \xi_i =0}  \frac{ \chi_{|\xi_2|=T+O(1)}}{ \langle \xi_1 \rangle^{\frac{1}{2}+}} \widehat{f}_1(\xi_1)\widehat{f}_2(\xi_2)\widehat{f}_3(\xi_3)d\xi \lesssim \prod_{i=1}^3 \|f_i\|_{L^2_x} \, .
\end{equation}
 An elementary calculation shows that
\begin{align*}
L.H.S. \, of \,  (\ref{50})
\lesssim \sup_{T \in{\mathbb N}} \| \chi_{|\xi|=T+O(1)} \ast \langle \xi \rangle^{-1-}\|^{\frac{1}{2}}_{L^{\infty}(\mathbb{R}^2)} \prod_{i=1}^3 \|f_i\|_{L^2_x} \lesssim \prod_{i=1}^3 \|f_i\|_{L^2_x}\, ,
\end{align*}
so that the desired estimate follows.\\
{\bf Proof of (\ref{2.4}):} This reduces to
\begin{align*}
 \int_* \frac{\widehat{u}_1(\tau_1,\xi_1)}{\langle  \xi_1\rangle^s \langle |\tau_1|-|\xi_1| \rangle^{\frac{1}{2}+}} 
\frac{\widehat{u}_2(\tau_2,\xi_2)}{\langle \xi_2 \rangle^s \langle |\tau_2| - |\xi_2|\rangle^{\frac{1}{2}+}} \frac{\langle \xi_3 \rangle^{s+\frac{1}{8}}
\widehat{u}_3(\tau_3,\xi_3)}{\langle \tau_3 \rangle^{\frac{1}{2}-}} d\xi d\tau
\lesssim \prod_{i=1}^3 \|u_i\|_{L^2_{xt}} \, .
\end{align*}
Assuming without loss of generality $|\xi_1| \le |\xi_2|$ we have to show
\begin{align*}
\int_* \frac{\widehat{u}_1(\tau_1,\xi_1)}{\langle  \xi_1\rangle^s \langle |\tau_1|-|\xi_1| \rangle^{\frac{1}{2}+}} 
\frac{\widehat{u}_2(\tau_2,\xi_2)}{ \langle |\tau_2| - |\xi_2|\rangle^{\frac{1}{2}+}} \frac{\langle \xi_3 \rangle^{\frac{1}{8}}
\widehat{u}_3(\tau_3,\xi_3)}{\langle \tau_3 \rangle^{\frac{1}{2}-}} d\xi d\tau 
\lesssim \prod_{i=1}^3 \|u_i\|_{L^2_{xt}} \, .
\end{align*}
Case 1: $|\tau_2| \ll |\xi_2|$.\\
We reduce to
\begin{align*}
 \int_* \frac{\widehat{u}_1(\tau_1,\xi_1)}{\langle  \xi_1\rangle^s \langle |\tau_1|-|\xi_1| \rangle^{\frac{1}{2}+}} 
\frac{\widehat{u}_2(\tau_2,\xi_2)}{ \langle  \xi_2\rangle^{\frac{3}{8}+}} \frac{
\widehat{u}_3(\tau_3,\xi_3)}{\langle \tau_3 \rangle^{\frac{1}{2}-}} d\xi d\tau 
\lesssim \prod_{i=1}^3 \|u_i\|_{L^2_{xt}} \, .
\end{align*}
This follows from
\begin{align*}
\Big|\int v_1 v_2 v_3 dx dt \Big| & \lesssim \|v_1\|_{L^6_x L^{2+}_t} \|v_2\|_{L^3_x L^2_t} \|v_3\|_{L^2_x L^{\infty -}_t} \\
&\lesssim \|v_1\|_{X^{\frac{1}{6}+,\frac{1}{2}+}_{|\tau|=|\xi|}} \|v_2\|_{X^{\frac{1}{3},0}_{|\tau|=|\xi|}} 
\|v_3\|_{X^{0,\frac{1}{2}-}_{\tau =0}} \, ,
\end{align*}
where we used (\ref{18}) for the first factor and Sobolev for the others. Obviously here is some headroom left.\\
Case 2: $|\tau_2| \gtrsim |\xi_2|$. 
In this case we use $\tau_1 + \tau_2 + \tau_3 =0$ to estimate
$$ 1 \lesssim \frac{\langle \tau_2 \rangle^{\frac{1}{2}-}}{\langle \xi_2 \rangle^{\frac{1}{2}-}} \lesssim \frac{\langle \tau_1 \rangle^{\frac{1}{2}-}}{\langle \xi_2 \rangle^{\frac{1}{2}-}} + \frac{\langle \tau_3 \rangle^{\frac{1}{2}-}}{\langle \xi_2 \rangle^{\frac{1}{2}-}} \, . $$
2.1: If the second term on the right hand side is dominant we have to show, using also $\langle \xi_3 \rangle^{\frac{1}{8}} \lesssim \langle \xi_2 \rangle^{\frac{1}{8}}$ :
\begin{align*}
 \int_* \frac{\widehat{u}_1(\tau_1,\xi_1)}{\langle \xi_1 \rangle^s \langle |\tau_1| - |\xi_1|\rangle^{\frac{1}{2}+} } 
\frac{\widehat{u}_2(\tau_2,\xi_2)}{\langle \xi_2\rangle^{\frac{3}{8}-} \langle |\tau_2|-|\xi_2|\rangle^{\frac{1}{2}+}}
\widehat{u}_3(\tau_3,\xi_3) d\xi d\tau
\lesssim\prod_{i=1}^3 \|u_i\|_{L^2_{xt}} \, ,
\end{align*}
which follows for $s>\frac{3}{8}$ by Prop. \ref{Prop. 1.1}.\\
2.2: If the first term on the right hand side is dominant we consider two subcases.\\
2.2.1: $|\tau_1| \lesssim |\xi_1|$ . 
We reduce to
\begin{align*}
 \int_* \frac{\widehat{u}_1(\tau_1,\xi_1) \langle \xi_1 \rangle^{\frac{1}{2}-s-}}{\langle |\tau_1| - |\xi_1|\rangle^{\frac{1}{2}+} } 
\frac{\widehat{u}_2(\tau_2,\xi_2)}{\langle \xi_2\rangle^{\frac{3}{8}-} \langle |\tau_2|-|\xi_2|\rangle^{\frac{1}{2}+}}
\frac{\widehat{u}_3(\tau_3,\xi_3)}{\langle \tau_3 \rangle^{\frac{1}{2}-}} d\xi d\tau
\lesssim\prod_{i=1}^3 \|u_i\|_{L^2_{xt}} \, .
\end{align*}
Using $|\xi_2| \ge |\xi_1|$ and $ s > \frac{3}{8} $ it suffices to show 
\begin{align*}
\int_* \frac{\widehat{u}_1(\tau_1,\xi_1)}{ \langle \xi_1 \rangle^{\frac{1}{8}+}\langle |\tau_1| - |\xi_1|\rangle^{\frac{1}{2}+} } 
\frac{\widehat{u}_2(\tau_2,\xi_2)}{\langle \xi_2\rangle^{\frac{1}{8}+} \langle |\tau_2|-|\xi_2|\rangle^{\frac{1}{2}+}}
\frac{\widehat{u}_3(\tau_3,\xi_3)}{\langle \tau_3 \rangle^{\frac{1}{2}-}} d\xi d\tau
\lesssim\prod_{i=1}^3 \|u_i\|_{L^2_{xt}} \, .
\end{align*}
This follows from
\begin{align*}
\Big| \int v_1 v_2 v_3 dx dt \Big| & \lesssim \|v_1\|_{L^4_x L^{2+}_t} 
\|v_2\|_{L^4_x L^{2+}_t} \|v_3\|_{L^2_x L^{\infty-}_t} \\
&\lesssim \|v_1\|_{X^{\frac{1}{8}+,\frac{3}{8}+}_{|\tau|=|\xi|}} \|v_2\|_{X^{\frac{1}{8}+,\frac{3}{8}+}_{|\tau|=|\xi|}    } \|v_3\|_{X^{0,\frac{1}{2}-}_{\tau =0}} \, , \\
\end{align*}
where we used (\ref{17}). \\
2.2.2: $|\tau_1| \gg |\xi_1|$ $\Rightarrow$ $\langle |\tau_1| - |\xi_1| \rangle^{\frac{1}{2}-} \sim \langle \tau_1 \rangle^{\frac{1}{2}-}$ . \\
We have to show
\begin{align*}
 \int_* \frac{\widehat{u}_1(\tau_1,\xi_1)}{ \langle \xi_1 \rangle^s} 
\frac{\widehat{u}_2(\tau_2,\xi_2)}{\langle \xi_2\rangle^{\frac{3}{8}-} \langle |\tau_2|-|\xi_2|\rangle^{\frac{1}{2}+}}
\frac{\widehat{u}_3(\tau_3,\xi_3)}{\langle \tau_3 \rangle^{\frac{1}{2}-}} d\xi d\tau
\lesssim\prod_{i=1}^3 \|u_i\|_{L^2_{xt}} \, .
\end{align*}
This follows from
\begin{align*}
\Big| \int v_1 v_2 v_3 dx dt \Big| & \lesssim \|v_1\|_{L^3_x L^2_t} 
\|v_2\|_{L^6_x L^{2+}_t} \|v_3\|_{L^2_x L^{\infty-}_t} \\
&\lesssim \|v_1\|_{X^{\frac{1}{3},0}_{|\tau|=|\xi|}} \|v_2\|_{X^{\frac{1}{6}+,\frac{1}{2}+}_{|\tau|=|\xi|}    } \|v_3\|_{X^{0,\frac{1}{2}-}_{\tau =0}} \, , \\
\end{align*}
where we used Sobolev for the first and last factor and (\ref{18}) for the second one. This completes the proof of (\ref{2.4}).\\
{\bf Proof of (\ref{2.3}):} We distinguish between low and high frequencies of $A_j^{df}$. For high frequencies, i.e. , $supp \, ({\mathcal F} A_j^{df}) \subset \{|\xi| \ge 1 \}$ , we obtain by (\ref{6}) and Prop. \ref{Prop. 1.1} for $s > \frac{3}{8}$ :
$$ \| |\nabla|^{\epsilon} A_j^{df} \|_{X^{s+\frac{3}{8}-\epsilon,\frac{1}{2}+}_{|\tau|=|\xi|}} \lesssim \| \langle \psi,\psi \rangle \|_{X^{s-\frac{5}{8},\frac{1}{2}+}_{|\tau|=|\xi|}} \lesssim \|\psi\|^2_{X^{s,\frac{1}{2}++}_{|\tau|=|\xi|}} \, . $$
In the low frequency case $|\xi_3| \le 1$ , where  $\langle \xi_1 \rangle \sim \langle \xi_2 \rangle$ , it suffices to show
\begin{align*}
 &\int_* \frac{\widehat{u}_1(\tau_1,\xi_1)}{ \langle \xi_1 \rangle^s \langle |\tau_1|-|\xi_1|\rangle^{\frac{1}{2}++}} 
\frac{\widehat{u}_2(\tau_2,\xi_2)}{\langle \xi_2\rangle^s \langle |\tau_2|-|\xi_2|\rangle^{\frac{1}{2}++}} \cdot\\
& \hspace{5em} \cdot\frac{\widehat{u}_3(\tau_3,\xi_3) \langle |\tau_3|-|\xi_3| \rangle^{\frac{1}{2}+} \langle \xi_3 \rangle^{s+\frac{3}{8}-\epsilon} \chi_{\{|\xi_3| \le 1 \}}}{|\xi_3|^{1-\epsilon}} 
d\xi d\tau
\lesssim\prod_{i=1}^3 \|u_i\|_{L^2_{xt}} \, .
\end{align*}
Assuming without loss of generality $\langle \tau_2 \rangle \le \langle \tau_1 \rangle$, we obtain
$ \langle |\tau_3|-|\xi_3| \rangle^{\frac{1}{2}+} \sim \langle \tau_3 \rangle^{\frac{1}{2}+} \lesssim  \langle \tau_1 \rangle^{\frac{1}{2}+} + \langle \tau_2 \rangle^{\frac{1}{2}+} \lesssim  \langle \tau_1 \rangle^{\frac{1}{2}+} $ .\\
If $|\tau_1| \gg |\xi_1|$ or $|\tau_1| \ll |\xi_1|$ , it suffices to show
\begin{align*}
 \int_* \frac{\widehat{u}_1(\tau_1,\xi_1)}{ \langle \xi_1 \rangle^s} 
\frac{\widehat{u}_2(\tau_2,\xi_2)}{\langle \xi_2\rangle^s \langle |\tau_2|-|\xi_2|\rangle^{\frac{1}{2}++}}
\frac{\widehat{u}_3(\tau_3,\xi_3) \chi_{\{|\xi_3| \le 1 \}}}{|\xi_3|^{1-\epsilon}} d\xi d\tau
\lesssim\prod_{i=1}^3 \|u_i\|_{L^2_{xt}} \, .
\end{align*}
This follows from
$$\Big|\int v_1 v_2 v_3 dx dt \Big| \lesssim \|v_1\|_{L^2_{xt}} \|v_2\|_{L^{\infty}_t L^2_x} \|v_3\|_{L^2_t L^{\infty}_x} \, , $$
which gives the desired result using $\dot{H}^{1-\epsilon}_x \hookrightarrow  L^{\infty}_x$ for low frequencies.\\
If $|\tau_1| \sim |\xi_1|$ , we use $\langle \xi_1 \rangle \sim \langle \xi_2 \rangle$ and reduce to
\begin{align*}
 \int_* \widehat{u}_1(\tau_1,\xi_1) 
\frac{\widehat{u}_2(\tau_2,\xi_2)}{\langle \xi_2 \rangle^{2s-\frac{1}{2}-} \langle \langle |\tau_2|-|\xi_2|\rangle^{\frac{1}{2}++}}
\frac{\widehat{u}_3(\tau_3,\xi_3) \chi_{\{|\xi_3| \le 1 \}}}{|\xi_3|^{1-\epsilon}} 
\lesssim\prod_{i=1}^3 \|u_i\|_{L^2_{xt}} \, ,
\end{align*}
which can be shown as before. We remark that we only used $s>\frac{1}{4}$ in the low frequency case.\\
{\bf Proof of (\ref{2.2}):} We even prove the estimate with $X^{s,-\frac{1}{2}++}_{|\tau|=|\xi|}$ replaced by $X^{s,0}_{|\tau|=|\xi|}$ on the left hand side. For high frequencies of $A^{df}_j$ we have to show
$$ \|A^{df}_j \alpha^j \psi \|_{X^{s,0}_{|\tau|=|\xi|}} \lesssim \|A^{df}\|_{X^{s+\frac{3}{8},\frac{1}{2}+}_{|\tau|=|\xi|}} \|\psi\|_{X^{s,\frac{1}{2}+}_{|\tau|=|\xi|}} \, , $$
which follows by Proposition \ref{Prop. 1.1}. For the low frequency case of $A^{df}_j$ it suffices to show
\begin{align*}
 \int_* \frac{\widehat{u}_1(\tau_1,\xi_1)}{ \langle \xi_1 \rangle^s \langle |\tau_1|-|\xi_1|\rangle^{\frac{1}{2}+}} 
\widehat{u}_2(\tau_2,\xi_2) \langle \xi_2\rangle^s 
\frac{\widehat{u}_3(\tau_3,\xi_3)  \chi_{\{|\xi_3| \le 1 \}}}{|\xi_3|^{\epsilon}\langle \|\tau_3|-|\xi_3| \rangle^{\frac{1}{2}+} \langle \xi_3 \rangle^{s+\frac{3}{8}-\epsilon}} d\xi d\tau\\
\lesssim\prod_{i=1}^3 \|u_i\|_{L^2_{xt}} \, .
\end{align*}
Using $\langle \xi_1 \rangle \sim \langle \xi_2 \rangle$ and $\langle \xi_3 \rangle \sim 1$ and $\dot{H}^{\epsilon}_x \hookrightarrow L^{\infty}_x$ for low frequencies this easily follows from the estimate
$$\Big|\int v_1 v_2 v_3 dx dt \Big| \lesssim \|v_1\|_{L^{\infty}_t L^2_x} \|v_2\|_{L^2_t L^2_x} \|v_3\|_{L^2_t L^{\infty}_x} \, . $$

This completes the proof of (\ref{2.1})-(\ref{2.4}). The property $\psi_{\pm} \in X^{s,1}_{\pm}[0,T]$ follows immediately from the proof of (\ref{2.1}) and (\ref{2.2}).
\end{proof}

\section{Proof of Theorem \ref{Theorem 1.2}}
\begin{proof}
Assume $s > \frac{19}{40}$ , say $s = \frac{19}{40} + \delta$ with $1 \gg \delta > 0$ . Let $\psi \in C^0([0,T],H^s)$ , $A_j \in C^0([0,T],H^{s+\frac{1}{8}})$ . \\
{\bf Claim 1:} $\psi_{\pm} \in X_{\pm}^{\frac{1}{4}+\alpha,\frac{1}{2}+}[0,T] $ , where $\alpha = \frac{1}{40} + \frac{3}{2}\delta-$ . \\
By Sobolev's multiplication law we obtain
$$ \|A_j \alpha^j \psi_{\pm}\|_{L^2([0,T],H^{2s-\frac{7}{8}})} \lesssim \|A\|_{C^0([0,T],H^{s+\frac{1}{8}})} \|\psi\|_{C^0([0,T],H^s)} T^{\frac{1}{2}} < \infty \, . $$ 
Thus $\psi_{\pm} \in X^{2s-\frac{7}{8},1}_{\pm}[0,T]$ . Interpolation with $\psi_{\pm} \in X_{\pm}^{s,0}[0,T] \subset C^0([0,T],H^s)$ gives $\psi_{\pm} \in X_{\pm}^{\frac{1}{4}+\frac{3}{2}s - \frac{11}{16}-,\frac{1}{2}+}[0,T] = X_{\pm}^{\frac{1}{4}+\alpha,\frac{1}{2}+}[0,T] $ .

We now iteratively improve the regularity of $\psi_{\pm}$ , $A^{cf}$ and $A^{df}$ in order to end up in a class where uniqueness holds by Theorem \ref{Theorem 1.1}.

Let us assume that $\psi_{\pm} \in X_{\pm}^{\min(\frac{1}{4}+\alpha_k,s),\frac{1}{2}+}[0,T]$ with $\alpha_k = \frac{1}{40} +\big(\frac{3}{2}\big)^k \delta - $ for some $k \in {\mathbb N}$. This was just shown for $k=1$ . If $\frac{1}{4}+\alpha_k \ge s$ , we obtain by (\ref{2.3}) and (\ref{2.4}) $|\nabla|^{\epsilon} A^{df}_j \in X^{s+\frac{3}{8}-\epsilon,\frac{1}{2}+}_{|\tau|=|\xi|}$ and $A^{cf}_j \in X^{s+\frac{1}{8},\frac{1}{2}+}_{\tau=0}[0,T]$ , so that uniqueness follows from Theorem \ref{Theorem 1.1}. 

Otherwise we now prove \\
{\bf Claim 2:} $A^{cf}_j \in X^{\min(\frac{1}{4}+2\alpha_k-,s+\frac{1}{8}),\frac{1}{2}+}_{\tau=0}[0,T]$ . \\
This reduces to
$$ \|\langle \psi,\alpha^j \psi \rangle \|_{X^{\frac{1}{4}+2\alpha_k-,-\frac{1}{2}++}_{\tau=0}} \lesssim \|\psi\|_{X^{\frac{1}{4}+ \alpha_k,\frac{1}{2}+}_{|\tau|=|\xi|}}^2 \, , $$
which is equivalent to
\begin{align*}
 \int_* \frac{\widehat{u}_1(\tau_1,\xi_1)}{\langle  \xi_1\rangle^{\frac{1}{4}+\alpha_k} \langle |\tau_1|-|\xi_1| \rangle^{\frac{1}{2}+}} 
\frac{\widehat{u}_2(\tau_2,\xi_2)}{\langle \xi_2 \rangle^{\frac{1}{4}+\alpha_k}\langle |\tau_2| - |\xi_2|\rangle^{\frac{1}{2}+}} \frac{\langle \xi_3 \rangle^{\frac{1}{4}+2\alpha_k -} 
\widehat{u}_3(\tau_3,\xi_3)}{\langle \tau_3 \rangle^{\frac{1}{2}-}} d\xi d\tau \\
\lesssim \prod_{i=1}^3 \|u_i\|_{L^2_{xt}} \, .
\end{align*}
Assuming without loss of generality $|\xi_1| \le |\xi_2|$ we reduce to
\begin{align*}
 \int_* \frac{\widehat{u}_1(\tau_1,\xi_1)}{\langle  \xi_1\rangle^{\frac{1}{4}+\alpha_k} \langle |\tau_1|-|\xi_1| \rangle^{\frac{1}{2}+}} 
\frac{\widehat{u}_2(\tau_2,\xi_2)}{\langle |\tau_2| - |\xi_2|\rangle^{\frac{1}{2}+}} \frac{\langle \xi_3 \rangle^{\alpha_k -} 
\widehat{u}_3(\tau_3,\xi_3)}{\langle \tau_3 \rangle^{\frac{1}{2}-}} d\xi d\tau 
\lesssim \prod_{i=1}^3 \|u_i\|_{L^2_{xt}} \, .
\end{align*}
Case 1: $|\tau_2| \ll |\xi_2|$ . \\
The left hand side is bounded by
\begin{align*}
& \int_* \frac{\widehat{u}_1(\tau_1,\xi_1)}{\langle  \xi_1\rangle^{\frac{1}{4}+\alpha_k} \langle |\tau_1|-|\xi_1| \rangle^{\frac{1}{2}+}} 
\frac{\widehat{u}_2(\tau_2,\xi_2)}{\langle \xi_2\rangle^{\frac{1}{2}-\alpha_k+}} \frac{
\widehat{u}_3(\tau_3,\xi_3)}{\langle \tau_3 \rangle^{\frac{1}{2}-}} d\xi d\tau \\
& \lesssim  \int_* \frac{\widehat{u}_1(\tau_1,\xi_1)}{\langle  \xi_1\rangle^{\frac{3}{4}} \langle |\tau_1|-|\xi_1| \rangle^{\frac{1}{2}+}} 
\widehat{u}_2(\tau_2,\xi_2) \frac{
\widehat{u}_3(\tau_3,\xi_3)}{\langle \tau_3 \rangle^{\frac{1}{2}-}} d\xi d\tau
\lesssim \prod_{i=1}^3 \|u_i\|_{L^2_{xt}} \, ,
\end{align*}
because by (\ref{16'''})
$$ \Big|\int v_1 v_2 v_3 dx dt \Big| \lesssim \|v_1\|_{L^{\infty}_x L^{2+}_t} \|v_2\|_{L^2_{xt}} \|v_3\|_{L^2_x L^{\infty -}_t} \lesssim \|v_1\|_{X^{\frac{1}{2}+,\frac{1}{2}+}_{|\tau|=|\xi|}} \|v_2\|_{X^{0,0}_{|\tau|=|\xi|}} \|v_3\|_{^{0,\frac{1}{2}-}_{\tau =0}}   $$
Case 2: $|\tau_2| \gtrsim |\xi_2|$ . \\
In this case we obtain
$$ 1 \lesssim \frac{\langle \tau_2 \rangle^{\frac{1}{2}-}}{\langle \xi_2 \rangle^{\frac{1}{2}-}} \lesssim \frac{\langle \tau_1 \rangle^{\frac{1}{2}-}}{\langle \xi_2 \rangle^{\frac{1}{2}-}} + \frac{\langle \tau_3 \rangle^{\frac{1}{2}-}}{\langle \xi_2 \rangle^{\frac{1}{2}-}} \, . $$
2.1: Concerning the second term we use $\langle \xi_3 \rangle \lesssim \langle \xi_2 \rangle$ and reduce to
\begin{align*}
 \int_* \frac{\widehat{u}_1(\tau_1,\xi_1)}{\langle  \xi_1\rangle^{\frac{1}{4}+\alpha_k} \langle |\tau_1|-|\xi_1| \rangle^{\frac{1}{2}+}} 
\frac{\widehat{u}_2(\tau_2,\xi_2)}{\langle \xi_2 \rangle^{\frac{1}{2}-\alpha_k+}\langle |\tau_2| - |\xi_2|\rangle^{\frac{1}{2}+}}  
\widehat{u}_3(\tau_3,\xi_3) d\xi d\tau 
\lesssim \prod_{i=1}^3 \|u_i\|_{L^2_{xt}} \, ,
\end{align*}
which follows from Proposition \ref{Prop. 1.1}. \\
2.2: Concerning the first term we consider two subcases. \\
2.2.1: $|\tau_1| \lesssim |\xi_1|$ . \\
This follows from
\begin{align*}
& \int_* \frac{\widehat{u}_1(\tau_1,\xi_1) \langle  \xi_1\rangle^{\frac{1}{4}-\alpha_k-}} {\langle |\tau_1|-|\xi_1| \rangle^{\frac{1}{2}+}} 
\frac{\widehat{u}_2(\tau_2,\xi_2)}{\langle \xi_2 \rangle^{\frac{1}{2}-\alpha_k}\langle |\tau_2|-|\xi_2| \rangle^{\frac{1}{2}+}}  
\frac{\widehat{u}_3(\tau_3,\xi_3)}{\langle \tau_3 \rangle^{\frac{1}{2}-}} d\xi d\tau \\
&\lesssim \int_* \frac{\widehat{u}_1(\tau_1,\xi_1)}{ \langle  \xi_1\rangle^{\frac{1}{8}+} \langle |\tau_1|-|\xi_1| \rangle^{\frac{1}{2}+}} 
\frac{\widehat{u}_2(\tau_2,\xi_2)}{\langle \xi_2 \rangle^{\frac{1}{8}+} \langle |\tau_2|-|\xi_2| \rangle^{\frac{1}{2}+}}  
\frac{\widehat{u}_3(\tau_3,\xi_3)}{\langle \tau_3 \rangle^{\frac{1}{2}-}} d\xi d\tau
\lesssim \prod_{i=1}^3 \|u_i\|_{L^2_{xt}} \, .
\end{align*}
Here we used $|\xi_1| \le |\xi_2|$ and $\alpha_k < \frac{1}{4}$, and the last step uses (\ref{17}) just as in the proof of (\ref{2.4}). \\
2.2.2: $|\tau_1| \gg |\xi_1|$  $\Longrightarrow$ $\langle |\tau_1|-|\xi_1| \rangle^{\frac{1}{2}-} \sim \langle \tau_1 \rangle^{\frac{1}{2}-}$ . \\
We reduce to
\begin{align*}
&\int_* \frac{\widehat{u}_1(\tau_1,\xi_1)}{ \langle  \xi_1\rangle^{\frac{1}{4}+\alpha_k}}  
\frac{\widehat{u}_2(\tau_2,\xi_2)}{\langle \xi_2 \rangle^{\frac{1}{2}-\alpha_k}\langle |\tau_2|-|\xi_2| \rangle^{\frac{1}{2}+}}  
\frac{\widehat{u}_3(\tau_3,\xi_3)}{\langle \tau_3 \rangle^{\frac{1}{2}-}} d\xi d\tau  
\lesssim \prod_{i=1}^3 \|u_i\|_{L^2_{xt}} \, .
\end{align*}
This is implied by
\begin{align*}
\Big|\int v_1 v_2 v_3 dx dt \Big| &
\lesssim \|v_1\|_{L^{\frac{8}{3}}_x L^2_t} \|v_2\|_{L^8_x L^{2+}_t} \|v_3\|_{L^2_x L^{\infty -}_t} \\
&\lesssim \|v_1\|_{X^{\frac{1}{4},0}_{|\tau|=|\xi|}} \|v_2\|_{X^{\frac{1}{4}+,\frac{1}{2}+}_{|\tau|=|\xi|}} \|v_3\|_{X^{0,\frac{1}{2}-}_{\tau =0}}  \, , 
\end{align*}
where we used Sobolev, (\ref{19}) and $\alpha_k < \frac{1}{4}$ . \\
{\bf Claim 3:} $|\nabla|^{\epsilon} A^{df}_j \in X^{\frac{1}{2}+2\alpha_k - \epsilon -,\frac{1}{2}+}_{|\tau|=|\xi|} $ . \\
For high frequencies we obtain
$$ \||\nabla|^{\epsilon} A^{df}_j\|_{X^{\frac{1}{2} +2\alpha_k-\epsilon-,\frac{1}{2}+}} \lesssim \| \langle \psi,\psi \rangle\|_{X^{-\frac{1}{2} +2\alpha_k-,\frac{1}{2}+}} \lesssim \|\psi\|^2_{X^{\frac{1}{4} +\alpha_k,\frac{1}{2}++}}  $$
by use of Proposition \ref{Prop. 1.1}. \\
The low frequency case can be handled as in the proof of (\ref{2.3}).

If after such an iteration step we obtained an $\alpha_k$ such that $\alpha_k > \frac{1}{8}$ , we obtain by  (\ref{2.1}) and (\ref{2.2}) combined with claim 2 and claim 3 the regularity $\psi_{\pm} \in X^{\frac{1}{4}+\alpha_k,\frac{1}{2}+}_{\pm}[0,T] \subset X^{\frac{3}{8}+,\frac{1}{2}+}_{\pm}[0,T]$ ,  $|\nabla|^{\epsilon} A^{df}_j \in X^{\frac{1}{2}+2\alpha_k - \epsilon -,\frac{1}{2}+}_{|\tau|=|\xi|}  \subset X^{\frac{3}{4}-\epsilon,\frac{1}{2}+}_{|\tau|=|\xi|}[0,T] $ and $A^{cf}_j \in X^{\frac{1}{4}+2\alpha_k-,\frac{1}{2}+}_{\tau=0}[0,T] \subset X^{\frac{1}{2}+,\frac{1}{2}+}_{\tau =0}[0,T]$ , where uniqueness holds by Theorem \ref{Theorem 1.1} and we are done.

If however $\alpha_k \le \frac{1}{8}$ we need a further iteration step. \\
{\bf Claim 4:} The following estimate holds:
$$\|A^{cf}_j \alpha^j \psi_{\pm}\|_{L^2_t(H_x^{3\alpha_k --})} \lesssim \|A^{cf}\|_{X^{\frac{1}{4}+2\alpha_k-,\frac{1}{2}+}_{\tau=0}} \|\psi_{\pm}\|_{X^{\frac{1}{4}+\alpha_k,\frac{1}{2}+}_{\pm}} \, . $$
This reduces to
\begin{align*}
& \int_* \frac{\widehat{u}_1(\tau_1,\xi_1)}{ \langle  \xi_1\rangle^{\frac{1}{4}+2\alpha_k-}\langle \tau_1 \rangle^{\frac{1}{2}+}} 
\frac{\widehat{u}_2(\tau_2,\xi_2)}{\langle \xi_2 \rangle^{\frac{1}{4}+\alpha_k}\langle |\tau_2|-|\xi_2| \rangle^{\frac{1}{2}+}} 
\widehat{u}_3(\tau_3,\xi_3) \langle \xi_3 \rangle^{3\alpha_k--} d\xi d\tau \\
& \hspace{25em}
\lesssim \prod_{i=1}^3 \|u_i\|_{L^2_{xt}} \, .  
\end{align*}
Case 1: $|\xi_1| \ge |\xi_2|$ $\Rightarrow$ $\langle \xi_3 \rangle \lesssim \langle \xi_1 \rangle$ . \\
Using $\alpha_k \le \frac{1}{4}$ it suffices to show
\begin{align*}
& \int_* \frac{\widehat{u}_1(\tau_1,\xi_1)}{\langle \tau_1 \rangle^{\frac{1}{2}+}} 
\frac{\widehat{u}_2(\tau_2,\xi_2)}{\langle \xi_2 \rangle^{\frac{1}{2}+}\langle |\tau_2|-|\xi_2| \rangle^{\frac{1}{2}+}} 
\widehat{u}_3(\tau_3,\xi_3) d\xi d\tau  
\lesssim \prod_{i=1}^3 \|u_i\|_{L^2_{xt}} \, ,  
\end{align*}
which holds by (\ref{20}). \\
Case 2: $|\xi_2| \ge |\xi_1|$ $\Rightarrow$ $\langle \xi_3 \rangle \lesssim \langle \xi_2 \rangle$ . \\
Here we use $\alpha_k \le \frac{1}{8}$. We only have to show
\begin{align*}
& \int_* \frac{\widehat{u}_1(\tau_1,\xi_1)}{\langle \xi_1 \rangle^{\frac{1}{2}+} \langle \tau_1 \rangle^{\frac{1}{2}+}} 
\frac{\widehat{u}_2(\tau_2,\xi_2)}{\langle |\tau_2|-|\xi_2| \rangle^{\frac{1}{2}+}} 
\widehat{u}_3(\tau_3,\xi_3) d\xi d\tau  
\lesssim \prod_{i=1}^3 \|u_i\|_{L^2_{xt}} \, ,  
\end{align*}
which follows from (\ref{21}). \\
{\bf Claim 5:} The following estimate holds:
$$\|A^{df}_j \alpha^j \psi_{\pm}\|_{L^2_t(H_x^{3\alpha_k --})} \lesssim \||\nabla|^{\epsilon} A^{df}\|_{X^{\frac{1}{2}+2\alpha_k-\epsilon-,\frac{1}{2}+}_{|\tau|=|\xi|}} \|\psi_{\pm}\|_{X^{\frac{1}{4}+\alpha_k,\frac{1}{2}+}_{\pm}} \, . $$
The case of high frequencies of $A^{df}_j$ this follows  from Proposition \ref{Prop. 1.1}, where we have to use our assumption $\alpha_k \le \frac{1}{8}$. In the case of low frequencies we can reduce to
\begin{align*}
&\int_* \frac{\widehat{u}_1(\tau_1,\xi_1) \chi_{\{|\xi_1| \le 1 \}}}{|\xi_1|^{\epsilon}\langle |\tau_1|-|\xi_1| \rangle^{\frac{1}{2}+}} 
\frac{\widehat{u}_2(\tau_2,\xi_2)}{\langle \xi_2 \rangle^{\frac{1}{4}+\alpha_k}\langle |\tau_2|-|\xi_2| \rangle^{\frac{1}{2}+}} 
\widehat{u}_3(\tau_3,\xi_3) \langle \xi_3 \rangle^{3\alpha_k --} d\xi d\tau \\
 &\lesssim \int_* \frac{\widehat{u}_1(\tau_1,\xi_1) \chi_{\{|\xi_1| \le 1 \}}}{|\xi_1|^{\epsilon}\langle |\tau_1|-|\xi_1| \rangle^{\frac{1}{2}+}} 
\frac{\widehat{u}_2(\tau_2,\xi_2)}{\langle \xi_2 \rangle^{\frac{1}{4}-2\alpha_k}\langle |\tau_2|-|\xi_2| \rangle^{\frac{1}{2}+}} 
\widehat{u}_3(\tau_3,\xi_3)  d\xi d\tau
\lesssim \prod_{i=1}^3 \|u_i\|_{L^2_{xt}} \, ,  
\end{align*}
which easily follows from the estimate
$$\big| \int v_1 v_2 v_3 dx dt \big| \lesssim \|v_1\|_{L^{\infty}_{xt}} \|v_2\|_{L^2_{xt}} \|v_3\|_{L^2_{xt}}  $$
for low frequencies of $v_1$, where we used again $\alpha_k \le \frac{1}{8}$.

We recall that $\alpha_k = \frac{1}{40}+\big(\frac{3}{2}\big)^k \delta \to \infty \, \, (k \to \infty) $ and $s=\frac{19}{40}+\delta$ with $1 \gg \delta >0$. Thus for some $k \in {\mathbb N}$ we have $\alpha_k \le \frac{1}{8}$ and $\alpha_{k+1} > \frac{1}{8}$.
Claim 4 and claim 5 imply that $\psi_{\pm} \in X_{\pm}^{\min(3\alpha_k-,s),1}[0,T]$. Interpolation with $\psi_{\pm} \in X_{\pm}^{s,0}[0,T] \supset C^0([0,T],H^s)$ gives $\psi_{\pm} \in X_{\pm}^{\min(\frac{3}{2} \alpha_k+\frac{s}{2}-,s),\frac{1}{2}+}[0,T]$ .
We notice that
$\frac{3}{2}\alpha_k + \frac{s}{2} = \frac{1}{4} + \big( \frac{1}{40} + \big(\frac{3}{2}\big)^{k+1} \delta) + \frac{\delta}{2} > \frac{1}{4}+\alpha_{k+1}$. Therefore
$\psi_{\pm} \in X_{\pm}^{\min(\frac{1}{4}+\alpha_{k+1},s),\frac{1}{2}+}[0,T] \subset X_{\pm}^{\frac{3}{8}+,\frac{1}{2}+}[0,T] $ ,
and by (\ref{2.3}) and (\ref{2.4}) we obtain $A^{cf}_j \in X^{\frac{1}{2}+,\frac{1}{2}+}_{\tau = 0}[0,T]$ and
$|\nabla|^{\epsilon} A^{df} \in  X_{\pm}^{\frac{3}{4}-\epsilon+,\frac{1}{2}+}[0,T]$ . 
In these spaces however uniqueness holds by Theorem \ref{Theorem 1.1}.
\end{proof}

\section{Proof of Theorem \ref{Theorem 1.1'} and Theorem \ref{Theorem 1.2'}}
\begin{proof}[Proof of Theorem \ref{Theorem 1.1'}]
By standard arguments we only have to show
$$ \| N(\psi_1,\psi_2,\psi_3) \|_{X^{s,-\frac{1}{2}++}_{\pm_4}} \lesssim \prod_{i=1}^3 \|\psi_i\|_{X^{s,\frac{1}{2}+}_{\pm_ i}} \, , $$
where $\pm_i$ $(i=1,2,3,4)$ denote independent signs.

By duality this is reduced to the estimates
$$ J:= \int \langle N(\psi_1,\psi_2,\psi_3),  \psi_4 \rangle dt\, dx \lesssim \prod_{i=1}^3 \|\psi_i\|_{X^{s,\frac{1}{2}+}_{\pm_i}} \|\psi_4\|_{X^{-s,\frac{1}{2}--}_{\pm_4}} \, . $$
By Fourier-Plancherel we obtain
$$ J = \int_* q(\xi_1,...,\xi_4) \prod_{j=1}^4 \widehat{\psi}_j(\xi_j,\tau_j) d\xi_1\, d\tau_1 ... d\xi_4\,d\tau_4 \, , $$
where * denotes integration over $\xi_1-\xi_2=\xi_4-\xi_3=:\xi_0$ and $\tau_1-\tau_2=\tau_4-\tau_3$ and
\begin{align*}
q =  \frac{1}{|\xi_0|^2} &[( \xi_{0_1}(\langle  \widehat{\psi}_1,\alpha_2 \widehat{\psi}_2 \rangle \langle \widehat{\psi}_3,\widehat{\psi}_4\rangle -  \langle  \widehat{\psi}_1,\widehat{\psi}_2 \rangle \langle \alpha_2 \widehat{\psi}_3,\widehat{\psi}_4\rangle) \\
& - \xi_{0_2}(\langle \widehat{\psi}_1,\alpha_1 \widehat{\psi}_2 \rangle \langle \widehat{\psi}_3,\widehat{\psi}_4\rangle -  \langle  \widehat{\psi}_1,\widehat{\psi}_2 \rangle \langle \alpha_1 \widehat{\psi}_3,\widehat{\psi}_4\rangle) ] \, .
\end{align*}
The specific structure of this term, namely the form of the matrices $\alpha_j$ plays no role in the following, thus the null structure is completely ignored.

We first consider the case $|\xi_0| \le 1$. In this case we estimate $J$ as follows:
\begin{align*}
 &\| \langle \nabla \rangle^{-s-1} |\nabla|^{-\frac{1}{2}} \langle  \psi_1,\alpha_i\psi_2 \rangle \|_{L^2_{xt}} \lesssim
\| \langle \psi_1,\alpha_i\psi_2 \rangle \|_{L^2_{x} H^{-s-1,\frac{4}{3}}_x} 
\lesssim \|\psi_1\|_{L^4_t H^s_x} \|\psi_2\|_{L^4_t H^{-s}_x} \, .
\end{align*}
In the last step we used 
\begin{align*}
\| f g \|_{H^s_x} \lesssim \|f\|_{H^s_x} \|g\|_{L^{\infty}_x} + \|f\|_{L^2_x} \|g\|_{H^{s,\infty}_x}
\lesssim \|f\|_{H^s_x} \|g\|_{H^{s+1,4}_x} \,
\end{align*}
which holds by the Leibniz rule for fractional derivatives and Sobolev's embedding theorem, and which is by duality equivalent to the required estimate
$$ \|fg\|_{H^{-s-1,\frac{4}{3}}_x} \lesssim \|f\|_{H^s_x} \|g\|_{H^{-s}_x} \, . $$
The same estimate holds for $\alpha_i = I$. Similarly we obtain 
\begin{align*}
 &\| \langle \nabla \rangle^{-s-1} |\nabla|^{-\frac{1}{2}} \langle \alpha_i \psi_3,\psi_4 \rangle \|_{L^2_{xt}}  
\lesssim \| \psi_3\|_{L^4_t H^s_x} \| \psi_4\|_{L^4_t H^{-s}_x} 
\end{align*}
for arbitrary matrices $\alpha_i$ ,
so that we obtain
\begin{align*}
J &  \lesssim \|\psi_1\|_{X^{s,\frac{1}{4}}_{\pm_1}} \|\psi_2\|_{X^{-s,\frac{1}{4}}_{\pm_2}} \|\psi_3\|_{X^{s,\frac{1}{4}}_{\pm_3}} \|\psi_4\|_{X^{-s,\frac{1}{4}}_{\pm_4}} \, , \end{align*}
which is more than enough.

From now on we assume $|\xi_0| \ge 1$. We obtain
\begin{align*}
|J|  \lesssim \sum_{j=1}^2 & \big( \|\langle  \psi_1,\alpha_j\psi_2 \rangle \|_{X_{|\tau|=|\xi|}^{s-\frac{1}{2},\frac{1}{4}}} \|\langle \psi_3,\psi_4\rangle\|_{X_{|\tau|=|\xi|}^{-s-\frac{1}{2},-\frac{1}{4}}} \\
& + \|\langle \psi_1,\psi_2 \rangle \|_{X_{|\tau|=|\xi|}^{s-\frac{1}{2},\frac{1}{4}}} \|\langle \alpha_j \psi_3,\psi_4\rangle\|_{X_{|\tau|=|\xi|}^{-s-\frac{1}{2},-\frac{1}{4}}} \big)\, .
\end{align*}
 By Proposition \ref{Prop. 1.1} with $s_0=\frac{1}{2}-s$ , $b_0=-\frac{1}{4}$ , $s_1=s_2 =s$ , $b_1=b_2=\frac{1}{2}+\epsilon$ for the first factors and $s_0 = s+\frac{1}{2}$ , $b_0 = \frac{1}{4}$ , $s_1=s$ , $s_2=-s$ , $b_1=\frac{1}{2}+\epsilon$ , $b_2=\frac{1}{2}-2\epsilon$ for the second factors
 we obtain 
$$ |J| \lesssim \prod_{j=1}^3 \|\psi_j\|_{X_{|\tau|=|\xi|}^{s,\frac{1}{2}+\epsilon}} \|\psi_4\|_{X_{|\tau|=|\xi|}^{-s,\frac{1}{2}-2\epsilon}} \, .$$
Using the embedding $X^{s,b}_{\pm} \subset X_{|\tau|=|\xi|}^{s,b}$ for $s \in {\mathbb R}$ and $b \ge 0$ we obtain the desired estimate.
\end{proof}
{\bf Remark:} The potentials are completely determined by $\psi$ and (\ref{1.10}). We have $A_{\mu} \sim |\nabla|^{-1} \langle \psi, \psi \rangle$ , so that for $s\le \frac{1}{2}$ :
$$ \|A_{\mu}\|_{\dot{H}^{2s}} \lesssim \|\langle \psi,\psi\rangle \|_{\dot{H}^{2s-1}} \lesssim  \|\langle \psi,\psi\rangle \|_{L^{\frac{1}{1-s}}} \lesssim \|\psi\|^2_{L^{\frac{2}{1-s}}} \lesssim \|\psi\|^2_{H^s} < \infty $$
and for $\frac{1}{2} < s < 1$ :
$$ \|A_{\mu}\|_{\dot{H}^{2s}} \lesssim \|\langle \psi,\psi\rangle \|_{\dot{H}^{2s-1}} \lesssim  \| \psi \|_{\dot{H}^{2s-1,\frac{2}{s}}} \|\psi\|_{L^{\frac{2}{1-s}}} \lesssim \|\psi\|_{H^s}^2 < \infty $$
as well as
$$ \|A_{\mu}\|_{\dot{H}^{\epsilon}} \lesssim \|\langle \psi,\psi\rangle \|_{\dot{H}^{\epsilon -1}} \lesssim \|\psi\|^2_{L^{\frac{4}{2-\epsilon}}} \lesssim \|\psi\|^2_{H^s} < \infty \, ,$$
thus we obtain for $0<\epsilon\ll 1$ and $s<1$ :
$$ A_{\mu} \in C^0([0,T],\dot{H}^{2s} \cap \dot{H}^{\epsilon}) \,. $$

\begin{proof}[Proof of Theorem \ref{Theorem 1.2'}]
We first show $\psi_{\pm} \in X^{0,1}_{\pm}[0,T]$. We have to prove
$$ \|N(\psi_1,\psi_2,\psi_3)\|_{L^2_t([0,T],L^2_x)} \lesssim \prod_{j=1}^3 \|\psi_j\|_{L^{\infty}_t([0,T], H^{\frac{1}{3}}_x)} \, , $$
where the implicit constant may depend on $T$ .
This follows from the estimate
\begin{align*}
\| |\nabla|^{-1} \langle  \psi_j,\alpha_i \psi_k \rangle \psi_3\|_{L^2_x} & \lesssim \| |\nabla|^{-1} \langle \psi_j,\alpha_i\psi_k \rangle \|_{L^6_x} \|\psi_3\|_{L^3_x} \lesssim \| \langle  \psi_j,\alpha_i\psi_k \rangle \|_{L^{\frac{3}{2}}_x} \|\psi_3\|_{L^3_x} \\
& \lesssim  \|\psi_j\|_{L^3_x} \|\psi_k\|_{L^3_x} \|\psi_3\|_{L^3_x} \lesssim \|\psi_j\|_{H^{\frac{1}{3}}_x} \|\psi_k\|_{H^{\frac{1}{3}}_x}\|\psi_3\|_{H^{\frac{1}{3}}_x} \, ,
\end{align*}
 and a similar estimate for the term $\| |\nabla|^{-1} \langle  \psi_j,\psi_k \rangle \alpha_i \psi_3\|_{L^2_x}$ .\\
Assume now $\psi \in C^0([0,T],H^{\frac{1}{3}+\epsilon})$ , $\epsilon > 0$. Then we have shown that $\psi_{\pm} \in X_{\pm}^{\frac{1}{3}+\epsilon,0}[0,T] \cap X^{0,1}_{\pm}[0,T]$. By interpolation we get $\psi_{\pm} \in X^{\frac{1}{4}+\frac{\epsilon}{4},\frac{1}{4}+\epsilon}_{\pm} [0,T] $ for $ \epsilon \ll 1.$ \\
Assume now that $\psi,\psi' \in C^0([0,T],H^{\frac{1}{3}+\epsilon})$ are two solutions of (\ref{1.11}),(\ref{1.12}), Then we have
\begin{align}
\label{**}
\sum_{\pm} \|\psi_{\pm}-\psi_{\pm}'\|_{X^{0,\frac{1}{2}+}_{\pm}[0,T]} &
\lesssim T^{0+} \sum_{\pm} \|N(\psi,\psi,\psi)-N(\psi',\psi',\psi')\|_{X_{\pm}^{0,-\frac{1}{2}++}[0,T]}
\\ \nonumber
& \lesssim T^{0+} \hspace{-1em} \sum_{\pm,\pm_1,\pm_2,\pm_3} \big(\|N(\psi_{\pm_1}-\psi_{\pm_1}',\psi_{\pm_2},\psi_{\pm_3})\|_{X_{\pm}^{0,-\frac{1}{2}++}[0,T]}\\ \nonumber
& \hspace{6em}+  \|N(\psi_{\pm_1}',\psi_{\pm_2}-\psi_{\pm_2}',\psi_{\pm_3})\|_{X_{\pm}^{0,-\frac{1}{2}++}[0,T]}\\ \nonumber
&\hspace{6em} +\|N(\psi_{\pm_1}',\psi_{\pm_2}',\psi_{\pm_3}-\psi_{\pm_3}')\|_{X_{\pm}^{0,-\frac{1}{2}++}[0,T]}\big)
\end{align}
Here $\pm$,$\pm_j$ $(j=1,2,3)$ denote independent signs. We want to show that for the first term the following estimate holds:
\begin{align}
\label{*}
\nonumber
J:= &\int \langle N(\psi_{\pm_1}-\psi_{\pm_1}',\psi_{\pm_2},\psi_{\pm_3}),\psi_{4} \rangle dx\,dt \\ 
&\lesssim \|\psi_{\pm_1}-\psi_{\pm_1}'\|_{X^{0,\frac{1}{2}+}_{\pm_1}} \|\psi_{\pm_2}\|_{X^{\frac{1}{4}+\frac{\epsilon}{4},\frac{1}{4}+\epsilon}_{\pm_2}} 
\|\psi_{\pm_3}\|_{X^{\frac{1}{4}+\frac{\epsilon}{4},\frac{1}{4}+\epsilon}_{\pm_3}}
\|\psi_{4}\|_{X^{0,\frac{1}{2}--}_{\pm_4}} \, .
\end{align}
We consider the case $|\xi_0|\le 1$ first. Similarly as in the proof of Theorem \ref{Theorem 1.1'} we obtain
$$ |J| \lesssim \|\psi_{\pm_1} - \psi_{\pm_1}'\|_{X^{-\frac{1}{4}-\frac{\epsilon}{4},\frac{1}{4}}_{\pm_1}}  \|\psi_{\pm_2}\|_{X^{\frac{1}{4}+\frac{\epsilon}{4},\frac{1}{4}}_{\pm_2}}
\|\psi_{\pm_3}\|_{X^{\frac{1}{4}+\frac{\epsilon}{4},\frac{1}{4}}_{\pm_3}} 
\|\psi_{4}\|_{X^{-\frac{1}{4}-\frac{\epsilon}{4},\frac{1}{4}}_{\pm_4}}\, , $$
which is more than sufficient. For $|\xi_0| \ge 1$ we obtain
\begin{align*}
|J| &\lesssim \sum_{j=1}^2 \big( \| \langle (\psi_{\pm_1} - \psi_{\pm_1}'),\alpha_j \psi_{\pm_2} \rangle \|_{X_{|\tau|=|\xi|}^{-\frac{1}{2},0}} \|\langle \psi_{\pm_3},\psi_4 \rangle\|_{X_{|\tau|=|\xi|}^{-\frac{1}{2},0}}\\ &\hspace{2.5em}
 + \| \langle \psi_{\pm_1} - \psi_{\pm_1}',\psi_{\pm_2} \rangle \|_{X_{|\tau|=|\xi|}^{-\frac{1}{2},0}} \|\langle \alpha_j \psi_{\pm_3},\psi_4 \rangle\|_{X_{|\tau|=|\xi|}^{-\frac{1}{2},0}}\big)\\
 & \lesssim \|\psi_{\pm_1} - \psi_{\pm_1}'\|_{X_{|\tau|=|\xi|}^{0,\frac{1}{2}+}} \|\psi_{\pm_2}\|_{X_{|\tau|=|\xi|}^{\frac{1}{4}+\frac{\epsilon}{4},\frac{1}{4}+\epsilon}}
 \|\psi_{\pm_3}\|_{X_{|\tau|=|\xi|}^{\frac{1}{4}+\frac{\epsilon}{4},\frac{1}{4}+\epsilon}}
 \|\psi_4\|_{X_{|\tau|=|\xi|}^{0,\frac{1}{2}--}} \, ,
\end{align*}
where we used Proposition \ref{Prop. 1.1} for the first factor with the choice
$s_0 = \frac{1}{2}$ , $b_0=0$, $s_1=0$ , $b_1=\frac{1}{2}+$ , $s_2=\frac{1}{4}+\frac{\epsilon}{4}$ , $b_2 = \frac{1}{4}+\epsilon$ and for the second factor with $s_0 = \frac{1}{2}$ , $b_0=0$ , $s_1= \frac{1}{4}+\frac{\epsilon}{4}$ , $b_1 = \frac{1}{4}+\epsilon$ , $s_2=0$ , $b_2= \frac{1}{2}--$. The embedding  $X^{s,b}_{\pm} \subset X_{|\tau|=|\xi|}^{s,b}$ for $b \ge 0$ gives (\ref{*}). The other terms in (\ref{**}) are treated similarly. We obtain
\begin{align*}
&\sum_{\pm} \|\psi_{\pm} - \psi_{\pm}'\|_{X^{0,\frac{1}{2}+}_{\pm}[0,T]} \\
& \lesssim T^{0+} \sum_{j=1}^2 \big( \|\psi_{\pm_j}\|^2_{X^{\frac{1}{4}+\frac{\epsilon}{4}, \frac{1}{4}+\epsilon}_{\pm_j}[0,T]} +\|\psi_{\pm_j}'\|^2_{X^{\frac{1}{4}+\frac{\epsilon}{4}, \frac{1}{4}+\epsilon}_{\pm_j}[0,T]} \big) \sum_{\pm} \|\psi_{\pm}-\psi_{\pm}'\|_{X^{0,\frac{1}{2}+}_{\pm}[0,T]} \, .
\end{align*}
We recall that $\psi_{\pm} \, , \, \psi'_{\pm} \in X^{\frac{1}{4}+\frac{\epsilon}{4},\frac{1}{4}+\epsilon}_{\pm} [0,T] $ , so that for sufficiently small $T$ 
this implies $\|\psi_{\pm} - \psi_{\pm}'\|_{X^{0,\frac{1}{2}+}_{\pm}[0,T]}=0$  , thus local uniqueness.
By iteration $T$ can be chosen arbitrarily.
\end{proof}

\end{document}